\documentclass[10pt]{article}
\usepackage{graphicx,amsmath,amsfonts,amssymb}
\usepackage{mathtools}
\usepackage{amsthm}
\usepackage[version=4]{mhchem}

\usepackage[paper=letterpaper,margin=2cm]{geometry}

\DeclareMathOperator{\rank}{rank}

\usepackage{tikz-cd}

\usepackage{xcolor}

\newcounter{comments}
\usepackage{comment}

\definecolor{jackcolor}{RGB}{0, 128, 32}

\theoremstyle{definition}
\newtheorem{thm}{Theorem}
\newtheorem{defn}[thm]{Definition}
\newtheorem{ex}[thm]{Example}

\providecommand{\keywords}[1]
{
  \small	
  \textbf{\textit{Keywords---}} #1
}

\title{Observability of complex systems via conserved quantities}
\author{Bhargav Karamched$^{a,b,c}$, Jack Schmidt$^d$, David Murrugarra$^d$}

\begin{document}

\maketitle

{\footnotesize
     \centerline{$^a$Department of Mathematics,
      Florida State University,}
  \centerline{Tallahassee, FL 32306-4510, USA}
}

{\footnotesize
     \centerline{$^b$Institute of Molecular Biophysics,
      Florida State University,}
  \centerline{Tallahassee, FL 32306-4510, USA}
}

{\footnotesize
     \centerline{$^c$Program in Neuroscience,
      Florida State University,}
  \centerline{Tallahassee, FL 32306-4510, USA}
}

{\footnotesize
     \centerline{$^d$Department of Mathematics,
      University of Kentucky,}
  \centerline{Lexington, KY 40506-0027, USA}
}

\begin{abstract}\noindent
Many systems in biology, physics, and engineering are modeled by nonlinear dynamical systems
where the states are usually unknown and only a subset of the state variables can be physically measured. Can we understand the full system from what we measure? In the mathematics literature, this question is framed as the observability problem. It has to do with recovering information about the state variables from the observed states (the measurements). In this paper, we relate the observability problem to another structural feature of many models relevant in the physical and biological sciences: the conserved quantity. For models based on systems of differential equations,
conserved quantities offer desirable properties such as dimension reduction which simplifies model analysis.
Here, we use differential embeddings to show that conserved quantities involving a set of special variables provide more flexibility in what can be measured to address the observability problem for systems of interest in biology.
Specifically, we provide conditions under which a collection of conserved quantities make the system observable.
We apply our methods to provide alternate measurable variables in models where conserved quantities have been used for model analysis historically in biological contexts.
\end{abstract}


\medskip

\section{Introduction}
A fundamental question in nonlinear dynamics is whether the entire state of a system can be inferred from measurements of a subset of outputs of the states that comprise the system. In the mathematics literature, this is referred to as the observability problem~\cite{aeyels1981generic,sontag2013mathematical}. Briefly, a dynamical system is called observable if one can obtain complete information about the internal state of the dynamical system from measurements of a subset of the outputs.

{In this paper, we consider dynamical systems of the form
\begin{equation}\label{eq:dyn_sys}
\frac{d\textbf{x}(t)}{dt} = f(\textbf{x}(t), \Lambda)
\end{equation}
with observable variables $\textbf{y} = g(\hat{\textbf{x}})$, where 
$f:\mathbb{R}^n \times \mathbb{R}^l\to\mathbb{R}^n$ and $g:\mathbb{R}^m\to\mathbb{R}^m$ are differentiable functions. Here, $\Lambda \in \mathbb{R}^l$ denotes inputs to the system~\eqref{eq:dyn_sys}, and it appears implicitly in the dynamics. We assume that we can only measure a subset of the state variables represented by $\hat{\textbf{x}}\in\mathbb{R}^m$ and the initial state $\textbf{x}_0\in\Omega_0\subset\mathbb{R}^n$ is unknown. For the rest of the paper, we suppress notation and write $ f(\textbf{x}(t)) = f(\textbf{x}(t),\Lambda$).}

{\begin{defn}
The system in Eq.~\eqref{eq:dyn_sys} is \textbf{observable} in $\Omega_0$ if there is a bijection between the initial states in $\Omega_0$ and the set of trajectories of the observed outputs $\textbf{y}(t)$ for $t\geq0$ \cite{kou1973observability}. We say that the system is \textbf{locally observable} at $x_0$ if there exists an open neighborhood $\Omega_0$ of $x_0$ for which the system is observable in $\Omega_0$.
\end{defn}}

Ascertaining which state variables to measure to completely understand a system is of central importance in physical and biological applications. Nonlinear dynamical systems have been used to model chemical reaction networks \cite{shinar2010structural,gunawardena2003chemical,feinberg2019foundations}, combustion reaction networks \cite{turns2011introduction,klippenstein2017theoretical,anderson2022analysis}, power grids \cite{witthaut2022collective,lopez2017assessment,osipov2018adaptive}, biophysical networks \cite{karamched2023stochastic,fazli2022network,hogan2021flipping,kim2012mechanism,del2015biomolecular}, epidemics \cite{weiss2013sir,kuznetsov1994bifurcation,ngonghala2020could}, and cancer \cite{hoffmann2020differential,plaugher2024cancer,komarova2023mutant}.  Observability of such dynamical systems is vital to constructively inform experimentalists and engineers what should be measured to optimize inference of the progress of their work. Most often, all variables involved are unable to be measured. Determining which outputs of a system should be measured to understand the full system is thus useful and essential for scientific and technological progress across disciplines.

For example, in determining the kinetic properties of an enzymatic reaction, one of the biochemical species must be measured to understand reaction rate. It is a challenging endeavor to simultaneously measure all constituents of the enzymatic reaction. Observable dynamical systems can inform biochemists of which species to track to understand the full system. Similarly, in an infection outbreak, observable dynamical systems can inform epidemiologists of what populations to track to optimally understand the dynamics of an epidemic. It is thus vital to develop mathematical theory and methods to ascertain whether a dynamical system is observable, and, if so, to determine which observables render the dynamical system observable.

\begin{figure}
    \centering
    \includegraphics[width=\textwidth]{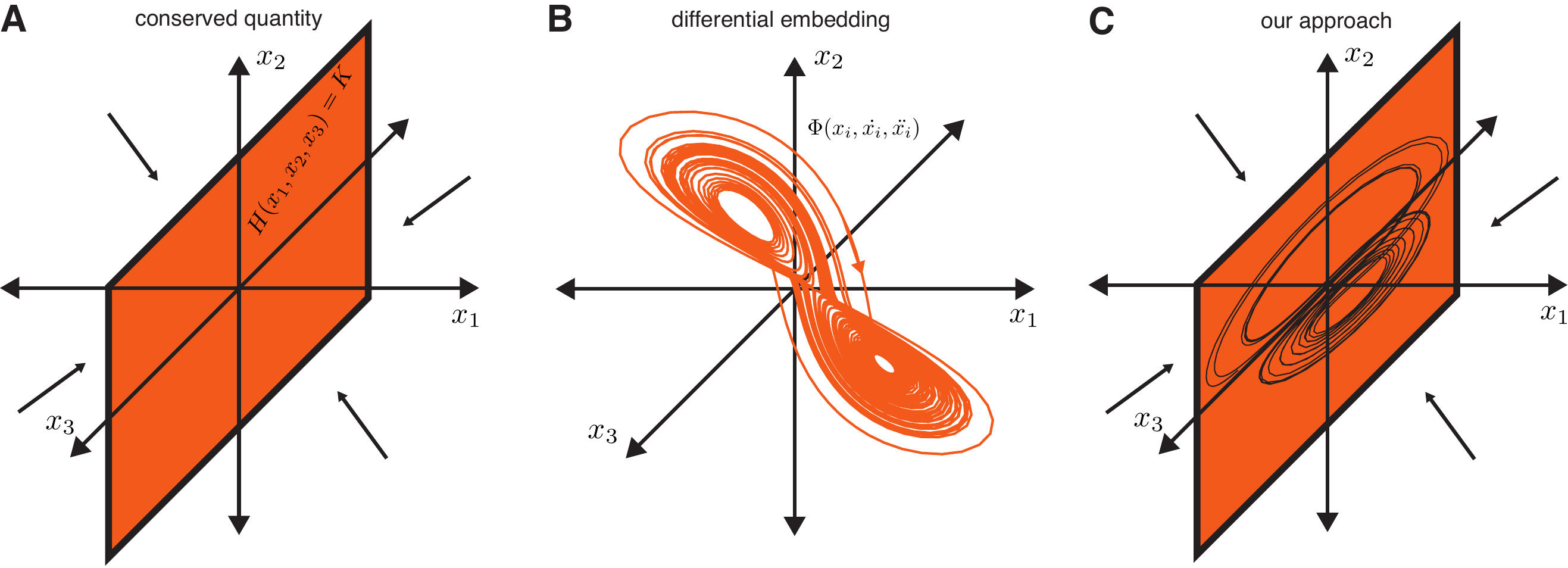}
    \caption{Schematic of the components and result of our work. (A) A conserved quantity {projects} dynamics of a dynamical system to a lower-dimensional submanifold. (B) A differential embedding is a transformation of the phase space of the original system. (C) By {projecting} the differential embedding onto the submanifold given by the conserved quantity, scalar observables that are not predicted to render the full system observable do make the system observable.}
    \label{fig:schematic}
\end{figure}

Several methods exist to determine whether a system is observable.  A longstanding method is to look at the Lie derivatives of the observables with respect to the governing nonlinear vector field and construct the Jacobian matrix of the Lie derivatives \cite{sontag2013mathematical,yano2020theory}. Parameter regimes where the Jacobian has full rank are those where the chosen observable renders the full system observable. {A related approach presented in \cite{letellier2005relation} considers a $k$-dimensional differential embedding $\Phi:\mathbb{R}^m\to \mathbb{R}^{km}$ given by $\Phi(\hat{\textbf{x}}) = (g(\hat{\textbf{x}}),\dot{g}(\hat{\textbf{x}}),\dots,g^{(k)}(\hat{\textbf{x}}))$ (derivatives with respect to time). The map $\Phi$ is locally invertible at $\textbf{x}_0$ if the Jacobian has full rank. That is, the map $\Phi$ is locally invertible at $\textbf{x}_0$ if
\begin{equation}\label{eq:rank}
\rank\left(\frac{\partial\Phi}{\partial\textbf{x}}\Big|_{\textbf{x}_0}\right) = n.
\end{equation}
The system in Eq.~\eqref{eq:dyn_sys} is locally observable at $x_0$ if and only if Eq.~\eqref{eq:rank} holds \cite{letellier2005relation}.}
 
{Another approach is the graphical approach. The graphical approach \cite{liu2013observability} associates a directed graph $\mathcal{G}$ to the system given by Eq.~\eqref{eq:dyn_sys}, where the nodes of $\mathcal{G}$ are $x_1$, $\dots$, $x_n$ and there is an edge from $x_i$ to $x_j$ if $x_j$ appears in the differential equation of $x_i$. The directed graph is partitioned into strongly connected components. A necessary and often sufficient condition for rendering Eq.~\eqref{eq:dyn_sys} observable is to observe one node in each strongly connected component that has no incoming edge.  
The graphical approach in \cite{liu2013observability} states that a necessary and sufficient condition for observability of the system in Eq.~\eqref{eq:dyn_sys} is to observe the source nodes in $\mathcal{G}$ and a variable in each strongly connected component of $\mathcal{G}$.} Finally there are methods based on a strongly positive definite condition and sensor selection based on optimization to ascertain observability\cite{kou1973observability,haber2017state}.

In this paper, we expand on the aforementioned results by considering the effect of another property of dynamical systems that manifests in several biological and physical applications:  the conserved quantity. A conserved quantity of a dynamical system is a function of the state variables that remains invariant in time:
{\begin{defn}
For $m\leq n$, a scalar-valued function $H:\mathbb{R}^m\to\mathbb{R}$ is a conserved quantity of Eq.~(\ref{eq:dyn_sys}) if, for all time and initial conditions,
\begin{equation}\label{eq:cq}
\frac{dH}{dt} = 0.
\end{equation}
\end{defn}
Note that if $m=n$, then the condition in Eq.~\ref{eq:cq} can also be stated as
\[
\nabla H\cdot \frac{d\textbf{x}(t)}{dt} = \nabla H\cdot f(\textbf{x}(t)) = 0,
\]
where $\nabla = (\frac{\partial}{\partial x_1},\dots,\frac{\partial}{\partial x_n})$. We can represent $\ell$ conserved quantities $H_1,\dots,H_\ell$ by using a function $G:\mathbb{R}^n\to\mathbb{R}^\ell$ where $G = (H_1,\dots,H_\ell)$. {Here, $G =$ constant.}}

{A typical use} of a conserved quantity is dimension reduction of the system under scrutiny. Because it defines a dependency between the state variables of system, the dynamics of the full system are {projected} to a submanifold of the phase space, thereby potentially simplifying analysis.

We show that conserved quantities in combination with differential embeddings provide a means to identify alternative observables in a system that render a system observable (see Fig.~\ref{fig:schematic}). {More specifically, we show that if a conserved quantity exists for a particular dynamical system, then observing the source node(s) of the related directed graph is \textit{no longer necessary.} The directed graph can be transformed via the conserved quantity, yielding a new set of variables that render the system observable. We demonstrate this in the following motivating example.}

\begin{ex}
  {SIR models are popular for describing dynamics of an infectious disease and for unveiling key biophysical parameters that govern the transition of a disease from dissipating in a population to persisting in an endemic state \cite{weiss2013sir,kuznetsov1994bifurcation,ngonghala2020could}. Such models are typically composed of three state variables: $S$ representing the number of susceptible individuals in a population, $I$ representing the number of infected individuals, and $R$ representing the number of recovered or removed individuals.  They have been shown to apply to more general settings as well by incorporating spatial and stochastic dynamics in their structure \cite{ji2014threshold,tornatore2005stability,milner2008sir, marques2022sir}. Furthermore, they have been used to study dissemination of information through a social network in a number of studies \cite{zhou2015influence,morsky2023impact}. Hence, SIR models form a crux of much of mathematical epidemiology literature.}

{One of the simplest SIR models describes the dynamics of an epidemic on a short timescale.  In such instances, the impact infection imparts on population dynamics vastly outweighs birth and death events, so birth and death terms do not manifest in the SIR dynamics. Because of this, the total number of individuals is invariant in time. Such a model is applicable, for example, in describing the dynamics and spread of the flu virus through a population \cite{hethcote2000mathematics}.}

{Here we investigate such a model.  Consider the following SIR model:
\begin{align}
\frac{dS}{dt} &= -\beta SI \nonumber\\
\frac{dI}{dt} &= \beta SI - \lambda I \label{eq:sir}\\
\frac{dR}{dt} &= \lambda I \nonumber,
\end{align}
where $S$ represents the susceptible population, $I$ represents the infected population, and $R$ represents the recovered population. The parameter $\beta$ quantifies the infectivity of the infectious disease under consideration; thus, the $\beta SI$ term captures the rate at which susceptible individuals become infected through contact with infected individuals. The parameter $\lambda$ quantifies the rate of recovery of an infected individual.  This system contains a conserved quantity, namely the total population. That is, $\forall t$, $S + I + R = N$ for a prescribed $N \in \mathbb{R}$. The total population is implicit in the right-hand side of Eq.~\eqref{eq:sir}, but is the input to the system. Because of this conserved quantity, Eq.~\eqref{eq:sir} can be reduced to a two-dimensional system that has the same equilibria and stability as the full system. Such a reduction greatly facilitates analysis.}

\begin{figure}[b!]
\includegraphics[width = \textwidth]{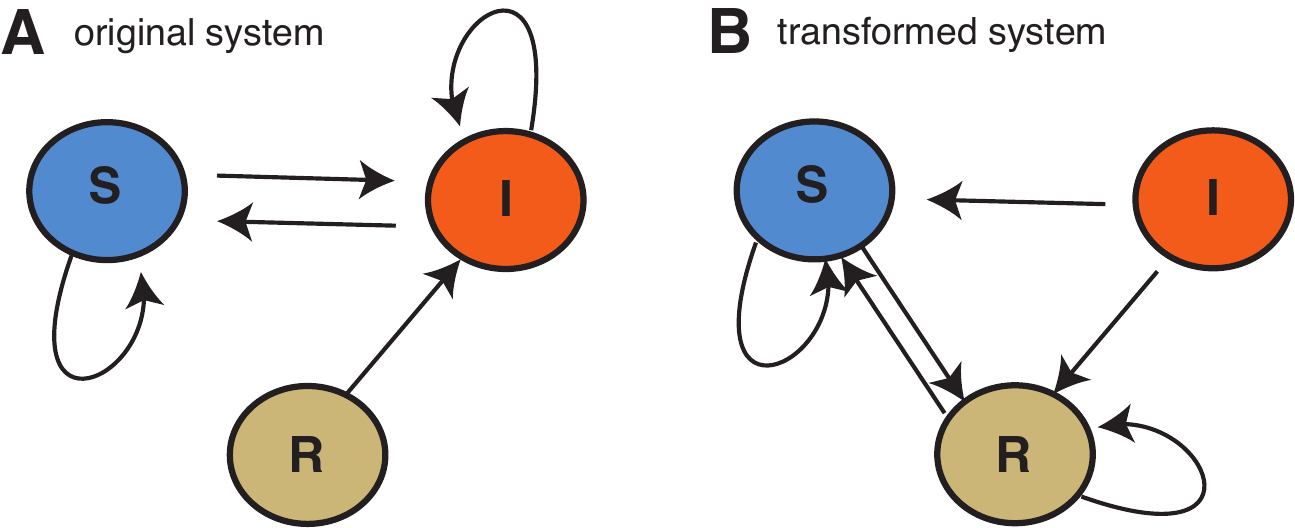}
\caption{Reaction graphs corresponding to the SIR system. (A) The original model. (B) The transformed model from invoking the transformation $I = N - S - R$.}
\label{fig:sir}
\end{figure}

{We depict the associated graph of this model in Fig.~\ref{fig:sir}A. According to the graphical approach in \cite{liu2013observability}, it is necessary to measure $R$ to make the system observable and that just measuring $I$ would not make the system observable. However, that conclusion would be wrong as we can get the information for $S$ and $R$ by measuring $I$ and using the conserved quantity $N$. Indeed, $S =  \tfrac{1}{\beta}\left( \dfrac{\dot I}{I} + \lambda\right)$ and $R=N-I-S = N - I -  \tfrac{1}{\beta}\left( \dfrac{\dot I}{I} + \lambda\right)$ expresses both $R$ and $S$ as functions of $I$, $\dot I = \frac{dI}{dt}$, and the conserved quantity $N$.}

{One could apply the graphical method to
\begin{align}
\frac{dS}{dt} &= -\beta SI \nonumber\\
\frac{dI}{dt} &= \beta SI - \lambda I \label{eq:sir2}\\
\frac{dN}{dt} &= 0\nonumber
\end{align}
to obtain that observing either $I$ and $N$ (or $S$ and $N$) suffices to recover all variables, including $R=N-I-S$, but the measurement of the unchanging quantity $N$ is practically quite different from measuring the varying quantity $I$.\\}
\end{ex}

We emphasize that existing methods for determining observability do not consider conserved quantities or how they impact the observability of a system. Existing theory pertaining to conserved quantities relates them to linear first integrals~\cite{szederkenyi2018analysis} or to dynamic conservation balances ~\cite{cameron2001process}. Therefore, the available methods are unable to detect alternative variables that render a system observable. For example, we show in this paper with an example that the graphical approach can miss alternate observables imputed by conserved quantities. {We do not claim that the graphical approach presented in ~\cite{liu2013observability} is incorrect. Our claim is that with knowledge of a conserved quantity, the necessary condition in the result therein can be relaxed. Indeed, our work here shows that knowledge of a conserved quantity provides a workaround for classical observability restrictions.}

Our approach is of interest to experimentalists and engineers because it provides a means to identify system outputs to measure that could reveal the internal state of the process being studied. Current methods for identifying observables {may fail to solve the observability problem because they may suggest observables that cannot be experimentally measured}. Our method provides flexibility in such scenarios.

Mathematically, our contribution is to append to the rich literature on observable and controllable systems.  We claim that if system dynamics can be {projected} to a submanifold that is inherent to the system, {alternative, potentially more practical, observables to what the original system suggested that render the full system observable can be ascertained.} Furthermore, our main result describes conditions for which submanifold to {project} dynamics if more than one are inherent to the system.

\section{Main Result}\label{sec:results}

For $m\leq n$, a subset of variables $\textbf{s}\in\mathbb{R}^m$ are called sufficient whenever observing these variables makes the system observable.
Next, we consider a partition of the variables in Eq.~\eqref{eq:dyn_sys}, $\textbf{x} = (\textbf{r},\textbf{s})$ where $\textbf{r}\in\mathbb{R}^{n-m}$ and $\textbf{s}\in\mathbb{R}^m$ is the set of sufficient variables.

Given a collection of conserved quantities $G:\mathbb{R}^n\to\mathbb{R}^\ell$, we describe its Jacobian using the partition above as follows:

\begin{equation}
\frac{\partial G}{\partial\textbf{x}}(\textbf{r},\textbf{s}) = \left[\begin{array}{ccc|ccc}
 \frac{\partial G_1}{\partial r_1}(\mathbf{r},\mathbf{s}) & \cdots & \frac{\partial G_1}{\partial r_{n-m}}(\mathbf{r},\mathbf{s}) &
 \frac{\partial G_1}{\partial s_1}(\mathbf{r},\mathbf{s}) & \cdots & \frac{\partial G_1}{\partial s_m}(\mathbf{r},\mathbf{s}) \\
 \vdots & \ddots & \vdots & \vdots & \ddots & \vdots \\
 \frac{\partial G_{\ell}}{\partial r_1}(\mathbf{r},\mathbf{s}) & \cdots & \frac{\partial G_{\ell}}{\partial r_{n-m}}(\mathbf{r},\mathbf{s}) &
 \frac{\partial G_{\ell}}{\partial s_1}(\mathbf{r},\mathbf{s}) & \cdots & \frac{\partial G_{\ell}}{\partial s_m}(\mathbf{r},\mathbf{s})
\end{array}\right]
= \left[\begin{array}{c|c} \frac{\partial G}{\partial\textbf{r}}(\textbf{r},\textbf{s}) & \frac{\partial G}{\partial\textbf{s}}(\textbf{r},\textbf{s}) \end{array}\right]
\end{equation}

Now we state our main result.

\begin{thm}
\label{main_thm}
Let
\begin{equation}\label{eq:thm_s}
\left\{
\begin{array}{ccc}
\dot{\textbf{x}}(t)  & = & f(\textbf{x}(t)) \\
\textbf{y} & = & g(\textbf{s})
\end{array}\right.
\end{equation}
be an observable system, where
$g:\mathbb{R}^m\to\mathbb{R}^m$.
If $G:\mathbb{R}^{n}\to\mathbb{R}^\ell$
is a collection of conserved quantities involving sufficient nodes $\textbf{s}$ and
other variables $\textbf{r}$ where
$\frac{\partial G}{\partial\textbf{s}}(\textbf{r},\textbf{s})$
is invertible and $\frac{\partial G}{\partial\textbf{r}}(\textbf{r},\textbf{s})$ full rank,
then $\exists\hspace{1pt}\hat{g}:\mathbb{R}^{n-m}\to\mathbb{R}^m$ such that the system
\begin{equation}\label{eq:thm_r}
\left\{
\begin{array}{ccc}
\dot{\textbf{x}}(t)  & = & f(\textbf{x}(t)) \\
\textbf{y} & = & \hat{g}(\textbf{r})
\end{array}\right.
\end{equation}
is observable.
\end{thm}
\begin{proof}
Since $G$ consists of conserved quantities,
$G = $constant.
Then, by the implicit function theorem, there is a function $\psi:\mathbb{R}^{n-m}\to\mathbb{R}^{m}$ such that $\textbf{s} = \psi(\textbf{r})$. Let $\hat{g} = g\circ \psi$. 
Since the system in Eq.~\eqref{eq:thm_s} is observable, the embedding $\Phi(\textbf{x}) = (g(\textbf{s}),g'(\textbf{s}),\cdots,g^{(k)}(\textbf{s}))$ is injective.
{Let's define the map $\hat{\Phi}(\textbf{r}) = (\hat{g}(\textbf{r}),\hat{g}'(\textbf{r}),\dots,\hat{g}^{(k)}(\textbf{r})) = (g\circ \psi(\textbf{r}),g\circ \psi'(\textbf{r}),\dots,g\circ \psi^{(k)}(\textbf{r}))$ as illustrated in the following diagram}, 
\begin{center}
\begin{tikzcd}
\mathbb{R}^{n-m} \arrow[r, "\psi" ] \arrow[rdd, dotted, "\hat{\Phi}"] & \mathbb{R}^{m} \arrow[dd, "\Phi"] \\
 &  \\
 & \mathbb{R}^{km}
\end{tikzcd}
\end{center}

\noindent {Thus, $\hat{\Phi} = \Phi\circ\psi$ and}
\[
\frac{\partial \hat{\Phi}(\textbf{r})}{\partial\textbf{r}} = \frac{\partial (\Phi\circ\psi)(\textbf{r})}{\partial\textbf{r}} = \frac{\partial \Phi(\textbf{s})}{\partial\textbf{s}} \frac{\partial \Psi(\textbf{r})}{\partial\textbf{r}},
\quad
\text{ but }
\quad
\frac{\partial \Psi(\textbf{r})}{\partial\textbf{r}} = -\left[\frac{\partial G}{\partial\textbf{s}}(\textbf{r},\textbf{s})\right]^{-1} \frac{\partial G}{\partial\textbf{r}}(\textbf{r},\textbf{s}).
\]
{Then,}

\[
\frac{\partial (\hat{\Phi})(\textbf{r})}{\partial\textbf{r}} = 
-\left[\frac{\partial G}{\partial\textbf{s}}(\textbf{r},\textbf{s})\right]^{-1} \frac{\partial G}{\partial\textbf{r}}(\textbf{r},\textbf{s})
\ \frac{\partial \Phi(\textbf{s})}{\partial\textbf{s}}.
\]

{Thus, $\hat{\Phi}$ is one-to-one which makes the system in Eq.~\eqref{eq:thm_r} is observable.}

\end{proof}

\section{Applications}\label{sec:applications}
Here we demonstrate that relatively simple systems of interest in biology containing conserved quantities are observable through the lense of Theorem~\ref{main_thm}. {A very simple example can be seen in Appendix~\ref{toy-example}.}

\subsection{Constant Population SIR Model}
In the following we first ascertain that Eq.~\eqref{eq:sir} is observable provided the observed state variable is $R(t)$. Then we construct the differential embedding map for the system and show that implementing the conserved quantity allows observing other state variables to render the full system observable.\\

\noindent \textit{The system given in Eq.~\eqref{eq:sir} is observable.} The observed variable will be $R(t)$. To determine whether or not Eq.~\eqref{eq:sir} is observable with $R(t)$ as the scalar observable, we must look at the Jacobian matrix associated with the Lie derivatives of this system~\cite{liu2013observability}.  Writing Eq.~\eqref{eq:sir} compactly as
$$
\frac{d\mathbf{X}}{dt} = f(\mathbf{X})
$$
with $\mathbf{X} \equiv (S, I, R)^T$ and $f(\mathbf{X}) = (-\beta SI, \beta SI - \lambda I, \lambda I)^T$, the Lie derivative of a scalar observable $y(t)$ is given by
$$
\mathcal{L}(y) = \frac{\partial y}{\partial t} + \sum_{i = 1}^3 f_i \frac{\partial y}{\partial \mathbf{X}_i}
$$
In accordance with the usual computations necessary for ascertaining observability, we compute
$$
\mathcal{L}^0 (R) = R \quad \quad \mathcal{L}^1 (R) = 2\lambda I \quad \quad \mathcal{L}^2(R) = 4\lambda (\beta S - \lambda)I
$$
and construct the associated Jacobian matrix given by
$$
\mathcal{J} = \left(\begin{array}{c}
\nabla \mathcal{L}^0(R)\\
\nabla \mathcal{L}^1(R)\\
\nabla \mathcal{L}^2(R)
\end{array}\right).
$$
That is, each row of the Jacobian matrix consists of a gradient vector of the Lie derivatives with respect to the state variables of the system. When $R(t)$ is observed, the Jacobian matrix is
\begin{equation}
\mathcal{J} = \left( \begin{array}{ccc}
0 & 0 & 1\\
0 & 2\lambda & 0\\
4\lambda \beta I & 4\lambda(\beta S - \lambda) & 0
\end{array} \right),
\label{eq:jacobian}
\end{equation}
which has full rank provided $\lambda \neq 0$, $\beta \neq 0$, and $I \neq 0$. Having full rank implies the system is observable.\\

\noindent \textit{The corresponding differential embedding is bijective.} Consider the embedding $\Phi(S,I,R) = (R, \dot{R}, \ddot{R})^T = (R, \lambda I, \lambda(\beta SI - \lambda I))^T$. This is bijective.

\begin{proof}
To prove injectivity, let $\Phi(S_1, I_1, R_1) = \Phi(S_2, I_2, R_2)$. Then
$$
\left(\begin{array}{c}
R_1\\
\lambda I_1\\
\lambda (\beta S_1 I_1 - \lambda I_1)
\end{array}\right) = \left(\begin{array}{c}
R_2\\
\lambda I_2\\
\lambda (\beta S_2 I_2 - \lambda I_2)
\end{array}\right)
$$
This clearly implies $R_1 = R_2$ and $I_1 = I_2$. Finally, we have $\lambda(\beta S_1I_1 - \lambda I_1) = \lambda(\beta S_2I_2 - \lambda I_2)$.  Since $I_1 = I_2$, this implies $S_1 = S_2$ and injectivity is proved. For surjectivity, take $\Phi(S,I,R) = (a,b,c)^T$ for some $(a,b,c)^T$ in the codomain of $\Phi$. Then clearly we can take $R = a$, $I = \frac{b}{\lambda}$ and $S = \frac{c}{\lambda\beta b} + \frac{\lambda}{\beta}$ as a preimage and surjectivity is proved.
\end{proof}
One subtle point is that we must constrain the codomain of $\Phi$ to be $\mathbb{R}^3\setminus\{(a,0,c):a,c \in \mathbb{R}\}$ for it to be surjective. This is completely consistent with the Jacobian in Eq.~\eqref{eq:jacobian}, which says that $I \neq 0$ is necessary for observability.  This is also consistent physically, since a situation where $I = 0$ is not particularly interesting when studying the spread of disease.

With the bijectivity of the differential embedding established, it is sufficient to consider the Jacobian of various embeddings to determine whether or not the observed variable renders the full system observable. From this perspective, we next show that observing $I$ in the absence of the conserved quantity does not render the system observable.

Consider now the differential embedding $\Psi(S,I,R) = (I, \dot{I}, \ddot{I})^T = (I, \beta SI - \lambda I, (\beta S-\lambda)^2I - \beta^2 S I^2)^T$. Clearly, $\Psi$ is not injective because the image of a point $(S, I, R)$ is agnostic to the value $R$ takes. \\

\noindent \textit{The conserved quantity renders $I$ a sufficient observable.} Consider the same differential embedding $\Psi = (I, \dot{I}, \ddot{I})^T$, but now let $I = N - S - R$, where we solve for $I$ in the conserved population equation $S + I + R = N \quad \forall t$. The corresponding differential equation system becomes
\begin{align*}
\frac{dS}{dt} &= -\beta S (N - S - R)\nonumber \\
\frac{dI}{dt} &= -\dot{S}-\dot{R} \\
\frac{dR}{dt} &= \lambda (N - S - R) \nonumber
\end{align*}
The corresponding differential embedding is
$$
\Psi = \left( \begin{array}{c}
I\\
(\beta S -\lambda)(N-S-R)\\
(\beta S - \lambda)^2(N-S-R) - \beta^2 S(N-S-R)^2
\end{array}\right)
$$
Then, the resulting Jacobian is
\begin{equation}
\left(\frac{\partial \Psi}{\partial \mathbf{X}}\right) = \left(\begin{array}{ccc}
0 & 1 & 0\\
\beta(N- 2S - R) & 0 & \lambda - \beta S  \\
F(S,R) & 0 & -(\beta S -\lambda)^2 + 2\beta^2S(N-S-R)
\end{array}\right )
\label{eq:jacobian_embedding}
\end{equation}
where $F(S,R) = (\beta S -\lambda)(2\beta(N-S-R)-\beta S + \lambda) + \beta^2(N-S-R)(3S+R-N)$. Again, provided $\beta \neq 0$ and $\lambda \neq 0$, $\left(\frac{\partial \Psi}{\partial \mathbf{X}}\right)$ has full rank and renders the system observable with the observed variable being $I(t)$.\\

\noindent \textit{Relating the two embeddings.} Since the system in Eq.~\eqref{eq:sir} is observable, the embedding $\Phi(S,I,R) = (R,\dot{R},\ddot{R})^T$ is bijective. Let $\hat{\Phi}(R,\dot{R},\ddot{R}) = \Psi = (I,\dot{I},\ddot{I})^T$ where $I = \psi(S,R) = N-R-S$.
Then, $\hat{\Phi}$ is a bijection such that the following diagram commutes.

\begin{center}
\begin{tikzcd}
\mathbb{R}^3 \arrow[r, "\Phi" ] \arrow[rdd, dotted] & \mathbb{R}^{3} \arrow[dd, "\hat{\Phi}"] \\
 &  \\
 & \mathbb{R}^{3}
\end{tikzcd}
\end{center}

\subsubsection{Relating to the Graphical Approach}
In summary, the preceding discussion says that Eq.~\eqref{eq:sir} is observable if the observed state is $R(t)$. This is consistent with what is obtained in the corresponding directed graph.

In the directed graph of the original SIR system, the only source node is $R$ (see Figure~\ref{fig:sir}A). The graphical approach for determining observability states that observing the source nodes of the directed graph of a system is necessary and sufficient to render the system observable. Consistent with the analysis in the previous section, observing $R$ rendered Eq.~\eqref{eq:sir} observable. Furthermore, in the original system, observing $I$ will \textit{not} render the system observable as $I$ is not a source node. However, it can be made into a source node by invoking the conserved quantity and transforming the system by setting $I = N - S - R$ (see Figure~\ref{fig:sir}B). In the transformed system, $I$ is the only source node, thereby making the system observable by observing $I$.  We note that if we make the transformation $S = N - R - I$, then $S$ will become the source node and it will be sufficient to observe $S$ to render the system observable. \\

A main takeaway is that the existence of the conserved quantity allows for more flexibility in tracking an epidemic from the perspective of the SIR model. Sans the conserved quantity, one can \textit{strictly} observe only $R$, the number of recovered individuals, to understand the full system. Simply observing only $S$ or only $I$ will not do the job.  However, the existence of the conserved quantity says that observing \textit{any one} of the state variables is sufficient to completely understand the system. Thus, trackers of epidemics have flexibility in measuring the epidemic by observing any one of the subpopulations---whichever one is easiest.

\subsection{Michaelis-Menten Kinetics}
The simplest enzyme kinetics are Michaelis-Menten kinetics, applied to enzyme-catalyzed reactions of one substrate and one product \cite{cook2007enzyme}. An enzyme E binds with its substrate S to form a complex ES which then dissociates into E and P, the product of the enzymatic reaction. The reaction network is as follows:
\begin{equation}
\ce{E + S <=>[{$k_1$}][$k_{-1}$] ES ->[$k_2$] E + P}
\label{eq:mm}
\end{equation}
where $k_1, k_{-1}, k_2$ are rate constants quantitating the corresponding reactions. Using the law of mass action, we can derive a model characterizing reaction~\eqref{eq:mm}.  Let $e \equiv [E], s \equiv [S], c \equiv [ES], \text{ and } p \equiv [P]$. Then we have \cite{keener2009mathematical}
\begin{equation}
\begin{aligned}
\frac{de}{dt} &= (k_{-1} + k_2)c - k_1es \\
\frac{ds}{dt} &= k_{-1}c - k_1es \\
\frac{dc}{dt} &= k_1es - (k_{-1} + k_2)c \label{eq:michael}\\
\frac{dp}{dt} &= k_2c
\end{aligned}
\end{equation}
There are two conserved quantities in this system:
\begin{equation}
\begin{aligned}
&e + c = E_0\\
&s + c + p = S_0
\label{eq:conserved_mm}
\end{aligned}
\end{equation}
where $E_0 \in \mathbb{R}$ represents the initial amount of enzyme in the system and $S_0 \in \mathbb{R}$ is the initial amount of substrate. Because these quantities are user-controlled, $E_0$ and $S_0$ can be construed as implicit inputs to the enzymatic system, Eq.~\eqref{eq:michael}. The two conserved quantities allow for dimensional reduction of system~\eqref{eq:michael} to a planar system
\begin{equation}
\begin{aligned}
\frac{ds}{dt} &= k_{-1}c - k_1(E_0 - c)s \\
\frac{dc}{dt} &= k_1(E_0 - c)s - (k_{-1} + k_2)c \label{eq:michael2}\\
\end{aligned}
\end{equation}
By rescaling $s, c,$ and $t$ and assuming that the concentration of substrate vastly outweighs the concentration of enzyme, we can derive the nondimensionalized system
\begin{equation}
\begin{aligned}
\frac{d\sigma}{d\tau} &= -\sigma + (1 - \eta + \sigma)\rho\\
\varepsilon\frac{d\rho}{d\tau} &= \sigma - (1+ \sigma)\rho\label{eq:michael3}\\
\end{aligned}
\end{equation}
where $\sigma \equiv k_1 s/(k_{-1} + k _2)$, $\rho \equiv c/E_0$, $\tau \equiv k_1 E_0 t$. We define the dimensionless parameters $\varepsilon \equiv E_0k_1/(k_{-1} +  k_2)$ and $\eta \equiv k_2/(k_{-1} + k_2)$ with $0 < \varepsilon \ll 1$. We can thereafter invoke the stationary state approximation \cite{meiss2007differential} and project onto the slow manifold \cite{rinzel1987formal} by assuming $\rho = \sigma(1 + \sigma)^{-1}$. Substituting this expression into the differential equation for $p$ then yields the classical Michaelis-Menten equation:
\begin{equation}
\frac{dp}{dt} = \frac{V_{\rm{max}}s}{K + s}
\label{eq:michael4}
\end{equation}
where $V_{\rm{max}} \equiv k_2 E_0$ is the fastest rate possible at which product P can be synthesized and $K \equiv (k_{-1} + k_2)/k_1$ is the dissociation constant.

The derivation and generalization of Eq.~\eqref{eq:michael4} to more complicated enzyme-substrate mechanisms are a central focus in the theoretical biochemical literature \cite{cook2007enzyme,sims2009aufbau}. While such derivations are important for the description of biochemical processes, they do not inform experimentalists of the ramifications of the theoretical models to the experiments themselves.

The conserved quantities in the Michaelis-Menten system confine the 4D dynamics to a two-dimensional submanifold, thereby allotting the desirable property of analytic tractabillity in the system. But what does the conserved quantity imply for experimentalists? Broadly, the existence of a conserved quantity consisting of variables that \textit{correspond to sources} in the directed graph representation \footnote{In the enzyme kinetics section of this paper, we will describe observability strictly through the graphical approach.} increases the number of variables that render the full system observable.

The reaction diagram for system~\eqref{eq:michael} is shown in Figure~\ref{fig:mm}A. The product P is the only source, implying that to understand the full system (i.e., to render the system observable), one must observe P. In an experimental setting, the kinetics of a given enzyme are measured and calculated from the observed dynamics of P. In a real setting, if P is easily measurable, then the situation at hand is no problem. However, in many situations, the product P is not directly measurable \cite{cook2007enzyme}. One must find an alternative to derive the kinetics of the corresponding enzymatic reaction. We demonstrate here that the presence of conserved quantities involving source terms allow for more freedom in observing the system.  We now systematically examine how the conserved quantities given in Eqs.~\eqref{eq:conserved_mm} alter the reaction diagram.

\begin{figure}
\includegraphics[width = \textwidth]{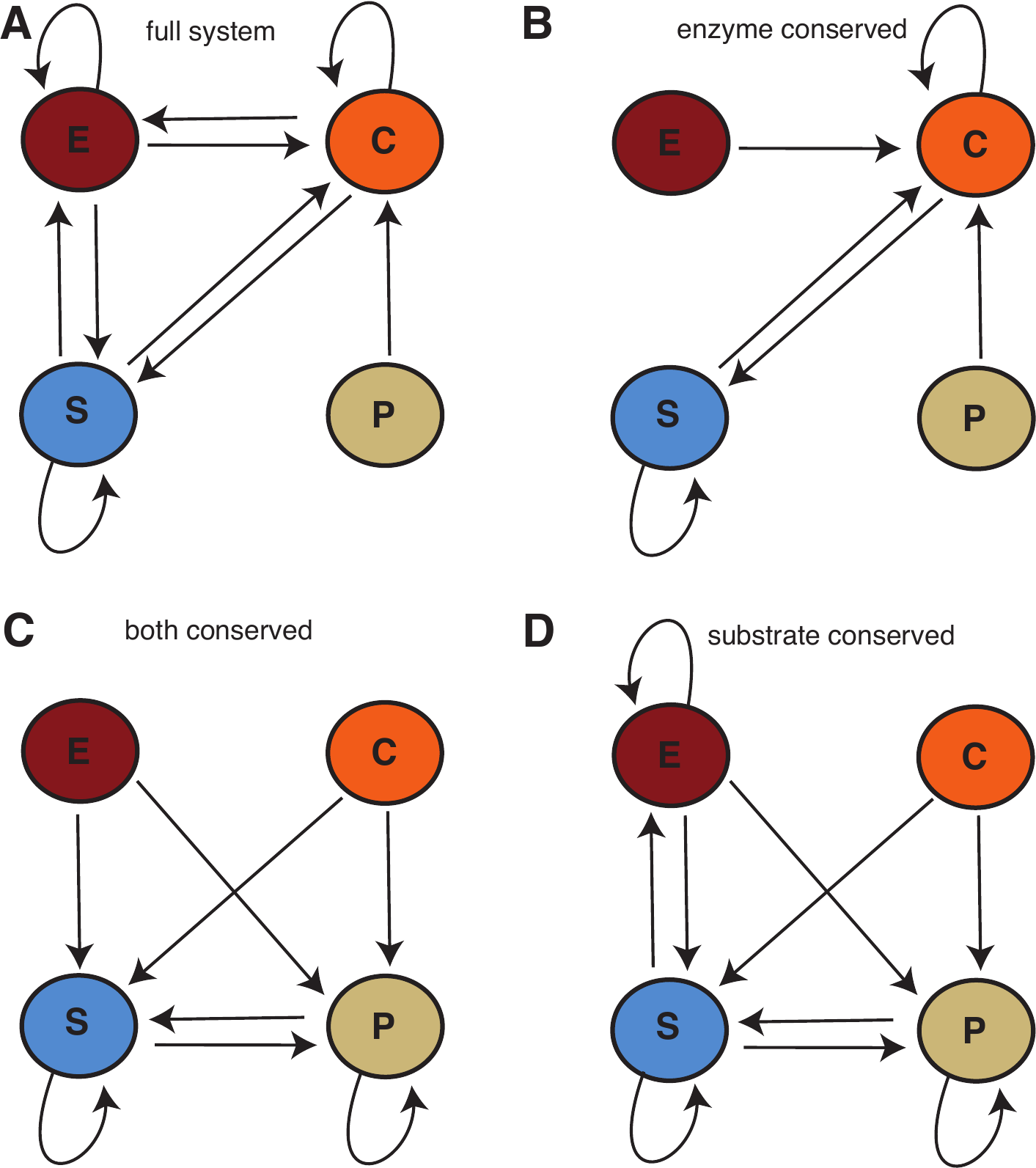}
\caption{Reaction graphs corresponding to the Michaelis-Menten system. (A) The full model. (B) The model with enzyme conservation imposed. (C) The model with enzyme conservation and substrate conservation imposed. (D) The model with substrate conservation imposed.}
\label{fig:mm}
\end{figure}

\subsubsection{Enzyme Conservation}
Let us suppose that we only impose enzyme conservation in the system. How does this alter the reaction diagram? In this case, we set $e = E_0 - c$, and the Michaelis-Menten system becomes
\begin{equation}
\begin{aligned}
\frac{de}{dt} &= -\frac{dc}{dt} \\
\frac{ds}{dt} &= k_{-1}c - k_1(E_0 - c)s \\
\frac{dc}{dt} &= k_1(E_0 - c)s - (k_{-1} + k_2)c \label{eq:michael_red_1}\\
\frac{dp}{dt} &= k_2c.
\end{aligned}
\end{equation}
Following our formalism for obtaining the corresponding reaction diagram, we obtain the diagram shown in Figure~\ref{fig:mm}B. It now has \textit{two} sources: E and P.  This means that to render the system observable, one must observe the dynamics of both E and P. Although the conserved quantity greatly simplifies mathematical analysis, the existence of this conserved quantity thus complicates the experimental setting. The issue arises because the imparted conserved quantity does not consist of the source from the full system, P.

We note that the reaction diagram would have a similar issue even if we took $c = E_0 - e$ in the conserved quantity.

\subsubsection{Substrate Conservation}
Now let us examine what happens when we impart substrate conservation. In this case, we set $c = S_0 - s - p$, rendering system~\eqref{eq:michael} as

\begin{equation}
\begin{aligned}
\frac{de}{dt} &= (k_{-1} + k_2)(S_0 - s - p) - k_1es \\
\frac{ds}{dt} &= k_{-1}(S_0 - s - p) - k_1es \\
\frac{dc}{dt} &= -\frac{ds}{dt}-\frac{dp}{dt} \label{eq:michael_red_2}\\
\frac{dp}{dt} &= k_2(S_0 - s - p)
\end{aligned}
\end{equation}
The corresponding reaction diagram is given in Figure~\ref{fig:mm}D. There is again only one source: C. All other nodes have incoming edges including self loops. The implication here is that now we need only observe C to understand the system. Furthermore, if we had set $s = S_0 - c - p$ instead, the only source in the resulting reaction diagram would be S, meaning we need to only observe S to render the system observable. The experimental implication is that one can observe the dynamics of any of S, C, or P to completely understand the system. Hence, if any of S, C, or P are measurable in a laboratory setting, the system can be understood. Thus, the conserved quantity consisting of the source node vastly expanded the number of state variables the we can measure to render the system completely observable.

\subsubsection{Enzyme and Substrate Conservation}
What happens if we impose conservation of both enzyme and substrate? Does this simplify the system further? In this case, we set $c = S_0 - s - p$ and $e = E_0 -c = E_0 - S_0 + s + p$. The system becomes
\begin{equation}
\begin{aligned}
\frac{de}{dt} &= \frac{ds}{dt} + \frac{dp}{dt} \\
\frac{ds}{dt} &= k_{-1}(S_0 - s - p) - k_1(E_0 - S_0 + s + p)s \\
\frac{dc}{dt} &= -\frac{ds}{dt} - \frac{dp}{dt} \label{eq:michael_red_3}\\
\frac{dp}{dt} &= k_2(S_0 - s - p)
\end{aligned}
\end{equation}
The corresponding reaction diagram is shown in Figure~\ref{fig:mm}C. Again, the diagram depicts two source nodes (E and C), implying one must observe both C and E to understand the system. This is, of course, incorrect.

The above analysis brings to light an important point: one must not conclude that theoretical conserved quantities imply positive experimental ramifications. Indeed, if one only analyzed the model with both substrate and enzyme conservation, they would conclude that one must observe two state variables to understand the enzymatic system. Conserved quantities that do not include source node state variables do not inform the observability of the system. The conserved quantity $s + c + p = S_0$, on the other hand, yields a correct interpretation of observability. Namely, any one of the terms involved in the conserved quantity can be observed to understand the system.

\section{Conclusions}\label{sec:conclusions}
We summarize the main contributions of this manuscript as follows. Most generally, we have proved a theorem conveying that observable dynamical systems with conserved quantities that involve source nodes in the corresponding directed graph representation of the system can be recast so that many more system outputs than originally thought could be observed to render the system observable. We used differential embeddings to prove this. In effect, we generalized the observability criteria provided by the graphical approach and the rank-based approach of differential embeddings.

Our approach has important implications for physical and biological sciences. Namely, we argue that systems with conserved quantities exhibit more flexibility in what must be observed for the full system to be understood. We demonstrate this with two concrete biological examples with conserved quantities:  the constant population SIR model and the classical Michaelis-Menten system for enzymatic reactions. For the former model, the original system necessitates observation of $R(t)$ to render the system observable.  However, the conserved quantity allows any one of $S, I,$ or $R$ to be observed for the system to be observable.  Similarly, the classical Michalis-Menten system requires observation of the product, $P(t)$, to render the system observable. The appropriate conserved quantity allows for product, substrate, or enzyme-substrate complex to be observed for the full system to be understood. Such flexibility can be the difference between success and failure in experimental settings.

For dynamical systems exhibiting multiple conserved quantities, our method identifies the `correct' submanifold of phase space to which dynamics should be {projected} to obtain alternative observables that render the full system observable. Only conserved quantities that incorporate source nodes of the associated directed graph of the dynamical system can yield other outputs of the system that render the dynamical system observable. {Employing conserved quantities that do not involve source nodes yield no benefit from the observability perspective. Moreover, conserved quantities that are hidden in systems but provide no physical interpretation can also be unavailing (see Appendix~\ref{not-useful} for an example).}

Mathematically, we contribute to the rich mosaic of literature available on controllable and observable systems. Our method will be of interest because it expands upon and improves the popular methods given by the graphical approach and the rank-based differential embeddings approach.

\section{Acknowledgments}

The authors thank anonymous reviewers for their insightful comments that have improved the manuscript. D.M. was partially supported by a Collaboration grant (\# 850896) from the Simons Foundation and a grant (NSF: \# 2424633) from the National Science Foundation. B.K. would like to thank his son, Surya, for his unconditional love and boundless energy. 

\bibliographystyle{ieeetr}
\bibliography{references}

\clearpage
\appendix

\section{A very small example}
\label{toy-example}
As a toy example, for any nonzero number $a$, we examine the two-variable linear system
\begin{equation}\label{eq:toy}
\begin{aligned}
\frac{dR}{dt} &= ~a R \\
\frac{dS}{dt} &= -a R
\end{aligned}
\end{equation}

We analyze the toy example explicitly, using the algebraic method, using the graphical method, and then the updated graphical method with conserved quantities.

\subsubsection{Explicit}

Equation \ref{eq:toy} can be explicitly solved as $R = R_0 e^{at}$ and $\dot S = -a R_0 e^{at}$ so $S = S_0+R_0-R_0 e^{at}$.

We might observe $y = g(R,S) = R$ in order to find $y(0) = R_0$ and $\dot y(0) = a R_0$, but this does not reveal $S_0$ and therefore $S(t)$ could not be recovered.

We might observe $y = g(R,S) = S$ in order to find $y(0) = S_0$ and $\dot y(0) = -a R_0$. We can recover $R_0 = -\tfrac1a \dot y(0)$ and $S_0 = y(0)$. Hence we can recover $R(t)$ and $S(t)$.

In other words, $R$ is not a sufficient observable, but $S$ is a sufficient observable. There are no degenerate parameters (assuming $a\neq 0$) or initial conditions.

\[ \begin{array}{cccc}
\text{Observed} & R_0 & S_0 & \text{Degenerate} \\ \hline \\[-1ex]
R & y_0 & \text{unobservable} & \text{all} \\[2ex]
S & -\dfrac{1}{a} \dot y_0 & y_0 & \text{none} \\
\end{array} \]

\subsubsection{Algebraic}

As equation \ref{eq:toy} is linear, we can easily use the algebraic characterization of observability.

Let $\vec{x}(t) = \begin{bmatrix} R(t) \\ S(t) \end{bmatrix}$ and $A=\begin{bmatrix} a & 0 \\ -a & 0 \end{bmatrix}$ so that $\dot x = Ax$.

If we observe $y=g(R,S) = R$, then $C=\begin{bmatrix} 1 & 0 \end{bmatrix}$ and the observability matrix is $O = \begin{bmatrix} C \\ C A \end{bmatrix} = \begin{bmatrix} 1 & 0 \\ a & 0 \end{bmatrix}$ which does not have full rank, so we cannot recover $S$.

If we observe $y=g(R,S) = S$, then $C=\begin{bmatrix} 0 & 1 \end{bmatrix}$ and the observability matrix is $O = \begin{bmatrix} 0 & 1 \\ -a & 0  \end{bmatrix}$ which is full rank and $S$ is a sufficient observable.

Thus again $R$ is not a sufficient observable, but $S$ is a sufficient observable.

\subsubsection{Graphical}

Using the graphical approach, we get \rotatebox[origin=c]{270}{$\circlearrowright$} $R \leftarrow S$. Hence we expect that $S$ is a sufficient observable, and $R$ is not. This agrees with our previous analysis.

\subsubsection{Conserved}

We note that $h(R,S) = R+S$ is a conserved quantity. If we observe both it and $R$, then we can rewrite the system as:
\begin{equation}
\begin{aligned}
\frac{dR}{dt} &= -\frac{dS}{dt} \\
\frac{dS}{dt} &= -a (Q_0-S)
\label{eq:red_toy}
\end{aligned}
\end{equation}
with our modified graph $R \to S$ \rotatebox[origin=c]{90}{$\circlearrowright$}. Hence we expect that $R$ is a sufficient observable if the conserved quantity $Q_0$ is also available. We verify this explicitly:

If we observe $g(R,S)=R$, then we would know $R(0)=R_0$ and $\dot R(0) = a(Q_0 - S(0))$. We can solve for $S(0) = Q_0 - \tfrac1a \dot R(0)$. Hence we can recover $R_0$ and $S_0$ and thus all $R(t)$ and $S(t)$. Alternatively, $R(0) + S(0) = Q_0$ immediately yields $S(0) = Q_0 - R_0$.

Note that our formula for $S_0$ involves the conserved quantity $Q_0$, but in a very simple way.

In order to rule out other graphical results, we classify all conserved quantities. If $h(R,S)$ is any conserved quantity, then it must be constant on the trajectories $(R_0 e^{at},S_0+R_0-R_0 e^{at})$. Letting $at\to-\infty$, we get that $h(R,S) = h(0,S+R) = k(R+S)$ is a function $R+S$, so our conserved quantity $Q_0=R_0+S_0 = R(t)+S(t)$ is the most general possible.

The only graphical representations possible using the conserved method are:
\begin{itemize}
\item \rotatebox[origin=c]{270}{$\circlearrowright$} $R \leftarrow S$ with no conserved quantities observed, and
\item $R \to S$ \rotatebox[origin=c]{90}{$\circlearrowright$} with the conserved quantity $R_0+S_0$ observed.
\end{itemize}

\clearpage

\section{An example where conserved quantities are not useful}

\label{not-useful}

The Lotka-Volterra predator-prey example has parameters $R>0$, $D>0$, $B>0$, $M>0$ describing the prey Reproductive rate, the Death probability when prey and predator meet, the Benefit of a meeting of prey and predator to the predator, and the Mortality rate of the predator in the absence of prey. We let $r$ be the prey (rabbits), and $m$ be the predators (monsters).
\begin{equation}\label{eq:useless}
\begin{aligned}
\frac{dr}{dt} &= Rr-Drm \\
\frac{dm}{dt} &= Brm-Mm
\end{aligned}
\end{equation}

We again analyze repeatedly.

\subsubsection{Explicit} There is no explicit solution for $r(t)$ and $s(t)$. However, we can still provide explicit recovery formulas for $r_0$ and $m_0$, including the degenerate solution spaces.

If we observe $y=g(r,m) = r$, then we know $y_0 = r_0$ and $\dot y_0 = Rr_0-Dr_0 m_0$. If $r_0 \neq 0$, then we may recover $r_0 = y_0$ and $m_0 = \dfrac{R}{D} - \dfrac{1}{D} \dfrac{\dot y_0}{y_0}$. If $r_0 = 0$, then $r(t) \equiv 0$, and $m(t) = m_0 e^{-Mt}$ with $m_0$ arbitrary. Hence when $r_0 = 0$, $r_0$ is not a sufficient obserable, but otherwise it is.

Similarly, if we observe $y = g(r,m) = m$, then we know $y_0 = m_0$ and $\dot y_0 = Br_0m_0 - Mm_0$. If $m_0 \neq 0$, then we may recover $r_0 = \dfrac{M}{B} + \dfrac{1}{B}\dfrac{\dot y_0}{y_0}$. If $m_0 = 0$, then $m(t) \equiv 0$ and $r(t) = r_0 e^{Rt}$ with $r_0$ arbitrary.

Hence each population is a sufficient observable, assuming it is not zero.

\[ \begin{array}{cccc}
\text{Observed} & r_0 & m_0 & \text{Degenerate} \\ \hline \\[-1ex]
r & y_0 & \dfrac{R}{D} - \dfrac{1}{D} \dfrac{\dot y_0}{y_0} & y_0 = 0 \\[2ex]
m & \dfrac{M}{B} + \dfrac{1}{B} \dfrac{\dot y_0}{y_0} & y_0 & y_0 = 0 \\
\end{array} \]

\subsubsection{Algebraic} While the system is not linear, we can linearize near a stationary point, and check there.

$\vec{x} = \begin{bmatrix} r - \dfrac{M}{B} \\[1ex] m - \dfrac{R}{D} \end{bmatrix}$ and $A = \begin{bmatrix} 0 & -\dfrac{DM}{B} \\[1ex] \dfrac{BR}{D} & 0 \end{bmatrix}$ and $\dot x \approx A \vec{x}$.

If we observe $y=g(r,m) = r$, then $C=\begin{bmatrix} 1 & 0 \end{bmatrix}$ and $O = \begin{bmatrix} 1 & 0 \\ 0 & -\dfrac{DM}{B} \end{bmatrix}$ is full rank, so $r$ is a sufficient observable as expected.

If we observe $y=g(r,m) = m$, then $C=\begin{bmatrix} 0 & 1 \end{bmatrix}$ and $O = \begin{bmatrix} 0 & 1 \\ \dfrac{BR}{D} & 0 \end{bmatrix}$ is full rank, so $m$ is a sufficient observable as expected.

Here $\dot x_1 - A_{1,:} \vec x = -D x_1 x_2$ and $\dot x_2 - A_{2,:} \vec x = B x_1 x_2$, so our results are valid in some open neighborhood of the stationary point. Indeed, our explicit results, show that our results are valid as long as $r_0 > 0$ and $m_0 > 0$.

\subsubsection{Graphical} We have a complete graph: \rotatebox[origin=c]{270}{$\circlearrowright$} $r \leftrightarrow m$ \rotatebox[origin=c]{90}{$\circlearrowright$}, so we expect each population to be a sufficient observable, agreeing with our previous analysis.

\subsubsection{Conserved} There is a conserved quantity, found by separation of variables on $\dfrac{dm}{dr} = \dfrac{Brm-Mm}{Rr-Drm}$ to get $\dfrac{Rr-Drm}{rm} dm = \dfrac{Brm-Mm}{rm} dr$ and so $\int (R\tfrac1m - D) dm = \int (N-M\tfrac1r) dr$ and so $R \ln(m) - Dm = Br-M\ln(r) + Q_0$. Hence $h(r,m) = R\ln(m) + M\ln(r) - Dm -Br$ is a conserved quantity.

Each trajectory (other than the stationary point $(r,m) = \left(\dfrac{M}{B},\dfrac{R}{D}\right)$) consists of the entire curve given by the implicit equation $h(r,m) = C$ for a unique $C > h\left(\dfrac{M}{B},\dfrac{R}{D}\right)$. In particular, every conserved quantity $h_2(r,m)$ must be equal to $k(C)=k(h(r,m))$ for some function $k$, and this conserved quantity is the most general conserved quantity.

However, using the conserved quantity does not affect our graphical methods. We would simply have to observe both $C$ and either $y=r$ or $y=m$. Since merely observing $y=r$ or $y=m$ was already sufficient, nothing is gained by additionally observing $C$.

This matches the intuitive feeling that the conserved quantity $h(r,m)$ is not of much use in solving this system. Its main use seems to be showing the trajectories are bounded.

\end{document}